\pdfoutput=1
\RequirePackage{ifpdf}
\ifpdf 
\documentclass[pdftex]{sigma}
\else
\documentclass{sigma}
\fi

\numberwithin{equation}{section}

\newtheorem{Theorem}{Theorem}[section]
\newtheorem{Lemma}[Theorem]{Lemma}
\newtheorem{Proposition}[Theorem]{Proposition}
 { \theoremstyle{definition}
\newtheorem{Remark}[Theorem]{Remark} }

\newcommand{\bk}{\backslash}

\newcommand{\pa}{\partial}

\newcommand{\Te}{Teichm\"{u}ller\:}

\newcommand{\ov}{\overline}

\newcommand{\vep}{\varepsilon}

\newcommand{\di}{\displaystyle}

\begin{document}
\allowdisplaybreaks

\newcommand{\arXivNumber}{2105.11074}

\renewcommand{\thefootnote}{}

\renewcommand{\PaperNumber}{097}

\FirstPageHeading

\ShortArticleName{Liouville Action for Harmonic Diffeomorphisms}

\ArticleName{Liouville Action for Harmonic Diffeomorphisms\footnote{This paper is a~contribution to the Special Issue on Mathematics of Integrable Systems: Classical and Quantum in honor of Leon Takhtajan.

~~\,The full collection is available at \href{https://www.emis.de/journals/SIGMA/Takhtajan.html}{https://www.emis.de/journals/SIGMA/Takhtajan.html}}}

\Author{Jinsung PARK}

\AuthorNameForHeading{J.~Park}

\Address{School of Mathematics, Korea Institute for Advanced Study,\\ 207-43, Hoegiro 85, Dong-daemun-gu, Seoul, 130-722, Korea}
\Email{\href{mailto:jinsung@kias.re.kr}{jinsung@kias.re.kr}}
\URLaddress{\url{http://newton.kias.re.kr/~jinsung/home.html}}

\ArticleDates{Received May 25, 2021, in final form October 27, 2021; Published online November 02, 2021}

\Abstract{In this paper, we introduce a Liouville action for a harmonic diffeomorphism from a compact Riemann surface to a compact hyperbolic Riemann surface of genus $g\ge 2$. We~derive the variational formula of this Liouville action for harmonic diffeomorphisms when the source Riemann surfaces vary with a fixed target Riemann surface.}


\Classification{14H60; 32G15; 53C43; 58E20}

\begin{flushright}
\begin{minipage}{65mm}
\it Dedicated to Professor Leon Takhtajan\\ on the occasion of his 70th birthday
\end{minipage}
\end{flushright}

\renewcommand{\thefootnote}{\arabic{footnote}}
\setcounter{footnote}{0}

\section{Introduction}

In mathematical physics, the Liouville action has been used as the action functional for the Liouville conformal field theory. In~mathematics, this was constructed in the works of Takhtajan--Zograf \cite{TZ88-1, TZ88-2}. They also proved several fundamental results of the Liouville action using the Teichm\"uller theory developed by Ahlfors--Bers. One of main results in \cite{TZ88-1, TZ88-2} is that the Liouville action is a K\"ahler potential of the Weil--Petersson symplectic 2-form on Teichm\"uller space.

One novelty of the works \cite{TZ88-1, TZ88-2} in the construction of the Liouville action
is the use of the projective structures on the Riemann surface. The projective structures determined by the geometric uniformizations in \cite{2, TZ88-1, TZ88-2} define the bounding noncompact hyperbolic 3-manifolds determined by those uniformizations. In~this geometric situation, the Liouville actions were proved to be the same as the renormalized volumes of the bounding hyperbolic 3-manifolds. We~refer to
\cite{Kras00, Kras-Sch08, 1, PT17, 2} for the relation of the Liouville action with the renormalized volume.

The harmonic map theory has been one of the main tools in the study of Teichm\"uller space.
The basic fact of this approach is that
there exists a unique harmonic diffeomorphism for the hyperbolic metric on the target Riemann surface in the homotopy class of an arbitrary homeomorphism between two compact Riemann surfaces.
Given a harmonic map, there is an associated Hopf differential, which is a holomorphic quadratic differential
on the source Riemann surface. In~\cite{Sam,W89}, it was proved that the Hopf differentials of harmonic diffeomorphisms give a natural parametrization of \Te space
fixing the source Riemann surface. Another asso\-ciated object to a harmonic map is its energy,
which can be considered as a functional on \Te space varying one of the source or target (hyperbolic) Riemann surface and fixing the other.

A modest motivation of this paper was to relate two objects~-- the Liouville action and the energy of harmonic diffeomorphisms~-- so that we may
have a certain object sharing the interesting properties of these two objects. To explain our approach to this problem, let us recall the construction of the Liouville action by Takhtajan--Zograf. In~\cite{TZ88-1,TZ88-2}, they crucially used a map
denoted by $J$ from the Poincar\'e half plane to the region of discontinuity for a~Kleinian group determined by a geometric uniformization. Then the main ingredient in the definition of the Liouville action is given by the pullback of the Poincar\'e metric by $J^{-1}$.

A possible generalization of the Liouville action could be achieved by using other map instead of $J^{-1}$. In~this paper, we develop this idea using a harmonic diffeomorphism which is canonically associated to the quasi-Fuchsian uniformization
for two marked compact hyperbolic Riemann surfaces.
By our construction, the Liouville action for a harmonic diffeomorphism contains the holomorphic energy of the harmonic map as a part.

As a first step to this study, we derive a variational formula for the Liouville action for harmonic diffeomorphisms. In~this formula, the variation of the Liouville action for harmonic diffeomorphisms is described mainly in terms of the Schwarzian derivative and the Hopf differential of harmonic diffeomorphisms. The precise variational formula is given in Theorem~\ref{t:final-thm}.
A~main part of the proof of this theorem is based on the variational formula of the Liouville action for a smooth family of conformal metrics on Riemann surfaces, which generalizes the work of Takhtajan--Teo in~\cite{2}.

Our approach in this paper may raise several related questions. One of them is a possibility to obtain another Liouville action for harmonic diffeomorphisms modifying the construction of this paper. We~take the term given by the holomorphic energy density in the pullback of the hyperbolic metric on the target Riemann surface by a harmonic diffeomorphism. But, we can also take
the Hopf differential part possibly among the parts of the pullback metric instead of our choice in this paper.
It is interesting to see how this different definition would provide us with a useful functional on \Te space.
See Remark~\ref{r:def-with-Hopf} for more detailed remark on this case.

Another natural question is the second variation of the Liouville action for diffeomorphisms. The second variation of the Liouville action defined by Takhtajan--Zograf in \cite{TZ88-1, TZ88-2} gives
the Weil--Petersson symplectic 2-form on Teichm\"uller space. On the other hand, the energy functional of harmonic diffeomorphisms varying the source Riemann surfaces with a fixed target hyperbolic Riemann surface is a strictly plurisubharmonic function on \Te space. This follows from that the Levi form given by the second variation of the energy functional is positive definite. For these, we refer to Tromba's book~\cite{Tromba}. Hence, as a common feature of the Liouville action and the energy functional of harmonic diffeomorphisms, one may wonder whether the Liouville action for diffeomorphisms would be a strictly plurisubharmonic function on Teich\-m\"iller space. The second variation of the Liouville action for harmonic diffeomorphisms and its possible applications will be studied elsewhere.

Finally let us explain the structure of this paper. We~start with the basic definitions and ter\-minologies for the Liouville action in Section~\ref{sec2}. This is a quick review of~\cite[Section~2]{2}. In~Section~\ref{s:harmonic}, we present the basics of the Liouville action for harmonic diffeomorphisms and derive its variational formula. In~Section~\ref{s:variation}, we prove the variational formula of the Liouville action for a smooth family of conformal metrics following~\cite[Section~4]{2}.

\section{Liouville action for quasi-Fuchsian groups}\label{sec2}

\looseness=1
Let us consider a compact Riemann surface $X$ with genus $g\geq 2$. Then the Riemann surface~$X$ can be realized by the \emph{Fuchsian uniformization}. It means that $X$ is given by
a quotient space~$\Gamma\bk\mathbb{U}$, where $\mathbb{U}$ is the upper half plane and $\Gamma$ is a marked, normalized Fuchsian group of the first kind. Here $\Gamma$ is a finitely generated cocompact discrete subgroup of $\text{PSL} (2, \mathbb{R})$ which has a~standard representation with $2g$ hyperbolic generators $\alpha_1, \beta_1, \ldots, \alpha_g, \beta_g$ satisfying the relation
\begin{gather*}
\alpha_1\beta_1\alpha_1^{-1}\beta_1^{-1}\cdots \alpha_g\beta_g\alpha_g^{-1}\beta_g^{-1}=I,
\end{gather*}
where $I$ is the identity element in $\Gamma$.

On the other hand, $X$ can be realized by the \emph{quasi-Fuchsian uniformization}. It means that~$X$ is given by a quotient space by a marked, normalized quasi-Fuchsian group $\Gamma \subset \text{PSL}(2,\mathbb{C})$.
This group has its region of discontinuity $\Omega\subset \hat{\mathbb{C}}$, which has two invariant components~$\Omega_1$ and~$\Omega_2$ separated by a quasi-circle $\mathcal{C}$. There exists a quasi-conformal homeomorphism $J_1$ of $\hat{\mathbb{C}}$ such that

\begin{enumerate}
\item[\textbf{QF1}]$J_1$ is holomorphic on $\mathbb{U}$ and $J_1(\mathbb{U})=\Omega_1$, $J_1(\mathbb{L})=\Omega_2$, $J_1(\mathbb{R})=\mathcal{C}$, where $\mathbb{U}$ and $\mathbb{L}$ are respectively the upper and lower half planes.
\item[\textbf{QF2}] $ \Gamma_1=J_1^{-1}\circ \Gamma\circ J_1$ is a marked, normalized Fuchsian group.
\end{enumerate}
 Since $J_1$ is holomorphic on $\mathbb{U}$, the Riemann surface $X=\Gamma_1\backslash \mathbb{U}$ is biholomorphic to the one given by the quasi-Fuchsian uniformization $\Gamma\backslash \Omega_1$. There is also a quasi-conformal homeomorphism~$J_2$ of $\hat{\mathbb{C}}$, holomorphic on $\mathbb{L}$ with a Fuchsian group $\Gamma_2=J_2^{-1}\circ \Gamma\circ J_2$ so that $\Gamma_2\backslash \mathbb{L}$ has the quasi-Fuchsian uniformization by $\Gamma\backslash\Omega_2$.

Let $\mathcal{A}^{-1,1}(\Gamma)$ be the space of Beltrami differentials for a quasi-Fuchsian group $\Gamma$, which is the Banach space of $\mu\in L^\infty(\mathbb{C})$ satisfying
\begin{gather*}
\mu(\gamma(z)) \frac{\overline{\gamma'(z)}}{\gamma'(z)} =\mu(z) \qquad \text{for all}\quad \gamma\in\Gamma
\end{gather*}
and
\begin{gather*}
\mu\, \big|_{\mathcal{C}} =0.
\end{gather*}
Denote by $\mathcal{B}^{-1,1}(\Gamma)$ the open unit ball in $\mathcal{A}^{-1,1}(\Gamma)$ with respect to the $\|\cdot\|_{\infty}$ norm. For each Beltrami coefficient
$\mu\in\mathcal{B}^{-1,1}(\Gamma)$, there exists a unique quasi-conformal map $f^\mu\colon\hat{\mathbb{C}} \to \hat{\mathbb{C}}$ satisfying the Beltrami equation
\begin{gather}\label{e:beltrami-eq}
f^\mu_{\bar{z}} =\mu f^\mu_z
\end{gather}
and fixing the points $0$, $1$ and $\infty$. Set $\Gamma^\mu= f^\mu \circ \Gamma\circ (f^\mu)^{-1}$ and define the \emph{deformation space of quasi-Fuchsian group} by
\begin{gather*}
\mathfrak{D}(\Gamma) =\mathcal{B}^{-1,1}(\Gamma)/ {\sim},
\end{gather*}
where $\mu\sim \nu$ if and only if $f^\mu=f^\nu$ on $\mathcal{C}$. The space $\mathfrak{D}(\Gamma)$ is a complex manifold of dimension $6g-6$. It is known that
\begin{gather*}
\mathfrak{D}(\Gamma) \simeq \mathfrak{T}(\Gamma_1)\times {\mathfrak{T}}(\Gamma_2),
\end{gather*}
where $\mathfrak{T}(\Gamma_i)$ is the \Te space of $\Gamma_i$ for $i=1,2$.
The deformation space $\mathfrak{D}(\Gamma,\Omega_1)$ is defined using the Beltrami coefficients supported on $\Omega_1$. By definition,
the space $\mathfrak{D}(\Gamma,\Omega_1)$ parametrizes all deformations of $X=\Gamma\backslash \Omega_1$ with the fixed Riemann surface $\Gamma\backslash \Omega_2$ so that
\begin{gather*}
\mathfrak{D}(\Gamma,\Omega_1) \simeq \mathfrak{T}(\Gamma_1).
\end{gather*}
{\sloppy Hence, it is possible to use the deformation space $\mathfrak{D}(\Gamma,\Omega_1)$ as the model of the Teichm\"{u}ller space~$\mathfrak{T}(\Gamma_1)$.
An advantage of this model is that one can use the holomorphic variation on $\mathfrak{D}(\Gamma,\Omega_1)$ given by the quasi-Fuchsian deformations.

}

In the following two subsections, we review the construction of the Liouville action for quasi-Fuchsian groups in~\cite{2}. We~refer to \cite[Section~2]{2} for more details.

\subsection{Homology construction}

Starting with a marked, normalized Fuchsian group $\Gamma$, the double homology complex $\mathsf{K}_{\bullet, \bullet}$ is defined as $\mathsf{S}_{\bullet}\otimes_{\mathbb{Z}\Gamma}\mathsf{B}_{\bullet}$, a tensor product over the integral group ring $\mathbb{Z}\Gamma$, where $\mathsf{S}_{\bullet}=\mathsf{S}_{\bullet}(\mathbb{U})$ is the singular chain complex of $\mathbb{U}$ with the differential $\partial'$, considered as a right $\mathbb{Z}\Gamma$-module, and $\mathsf{B}_{\bullet}=\mathsf{B}_{\bullet}(\mathbb{Z}\Gamma)$ is the standard bar resolution complex for $\Gamma$ with differential $\partial''$. The associated total complex $\text{Tot} \;\mathsf{K}$ is equipped with the total differential $\partial=\partial'+(-1)^p \partial''$ on $\mathsf{K}_{p,q}$.

There is a standard choice of the fundamental domain
$F\subseteq \mathbb{U}$ for $\Gamma$ as a non-Euclidean polygon with
$4g$ edges labeled by $a_k$, $a_k'$, $b_k'$, $b_k$; $1\leq k\leq g$ satisfying
$\alpha_k(a_k')=a_k$, $\beta_k(b_k')=b_k$, $1\leq k\leq g$. The orientation of the edges
is chosen such that
\begin{gather}\label{eq1}
\pa' F=\sum_{k=1}^{g}(a_k+b_k'-a_k'-b_k).
\end{gather}
Set $\pa' a_k=a_k(1)-a_k(0)$, $\pa' b_k=b_k(1)-b_k(0)$, so that $a_k(0)=b_{k-1}(0)$, $2\leq k \leq g$.

According to the isomorphism
$\mathsf{S}_{\bullet}\simeq\mathsf{K}_{\bullet,0}$, the fundamental domain $F$ is
identified with $F \otimes [\;] \allowbreak\in \mathsf{K}_{2,0}$. We~have $\pa'' F
= 0$, and it follows from \eqref{eq1} that
\begin{gather*}
\pa' F = \sum_{k=1}^g \big(\beta_k^{-1}(b_k) - b_k -\alpha_k^{-1}(a_k) +a_k\big) =\partial'' L,
\end{gather*}
where $L\in \mathsf{K}_{1,1}$ is given by
\begin{gather*} 
L= \sum_{k=1}^g \big(b_k\otimes[\beta_k] -a_k\otimes[\alpha_k]\big).
\end{gather*}
There exists $V\in\mathsf{K}_{0,2}$ such that
\begin{gather*}
 \pa ' L=\pa'' V.
\end{gather*}
 One can verify that it is given by
\begin{align*} 
V= {}&\sum_{k=1}^g \big(a_k(0)\otimes[ \alpha_k| \beta_k] - b_k(0)\otimes[\beta_k|\alpha_k] +
b_k(0)\otimes\big[\gamma_k^{-1}|\alpha_k\beta_k\big]\big)
\\
& -\sum_{k=1}^{g-1}
b_g(0)\otimes\big[\gamma_g^{-1}\cdots\gamma_{k+1}^{-1}|\gamma_k^{-1}\big],
\end{align*}
where $\gamma_k=[\alpha_k, \beta_k]=\alpha_k\beta_k\alpha_k^{-1}\beta_k^{-1}$.
Define
\begin{gather*}
\Sigma = F + L - V.
\end{gather*}
Then
\begin{gather*}
\pa \Sigma = 0.
\end{gather*}

Finally, we also define $W$ in the following way. Let $P_k$ be a $\Gamma$-contracting path (see~\cite[Definition~2.3]{2} for the precise definition of $\Gamma$-contracting) connecting 0 to $b_k(0)$. Then
\begin{gather*} 
W = \sum_{k=1}^g\! \big(P_{k-1}\otimes[\alpha_k|\beta_k] -
P_k\otimes[\beta_k|\alpha_k] +P_k\otimes[\gamma_k^{-1}|\alpha_k\beta_k]\big)
-\sum_{k=1}^{g-1}\!
P_g\otimes\big[\gamma_g^{-1}\cdots\gamma_{k+1}^{-1}| \gamma_k^{-1}\big].
\end{gather*}

If $\Gamma$ is a quasi-Fuchsian group, let $\Gamma_1$ be the Fuchsian group such that $\Gamma_1=J_1^{-1}\circ \Gamma\circ J_1$. The double complex associated with $\Omega_1$ and the group $\Gamma$ is a push-forward by the map $J_1$ of the double complex associated with $\mathbb{U}$ and the group $\Gamma_1$. Define
\begin{gather*}
\Sigma_1 =F_1 +L_1 -V_1,
\end{gather*}
where $F_1=J_1(F)$, $L_1=J_1(L)$, $V_1=J_1(V)$. Similarly we define
\begin{gather*}
\Sigma_2 =F_2 +L_2 -V_2.
\end{gather*}
Here $F_2=J_1(F')$, $L_2=J_1(L')$, $V_2=J_1(V')$, where the corresponding chains $F'$, $L'$, $V'$ in $\mathbb{L}$ are given by the complex conjugation of $F$, $L$, $V$ respectively.

\subsection{Cohomology construction}

The corresponding double complex in cohomology $\mathsf{C}^{\bullet, \bullet}$ is defined as $\mathsf{C}^{p, q}=\text{Hom}_{\mathbb{C}}\left(\mathsf{B}_q, \mathsf{A}^p\right)$, where~$\mathsf{A}^{\bullet}$ is the complexified de Rham complex on $\Omega_1$. The associated total complex $\text{Tot} \mathsf{C}$ is equipped with the total differential $D={\rm d}+(-1)^p\delta$ on $\mathsf{C}^{p, q}$, where ${\rm d}$ is the de Rham differential and $\delta$ is the group coboundary. The natural pairing $\langle \, , \, \rangle$ between $\mathsf{C}^{p, q}$ and $\mathsf{K}_{p, q}$ is given by the integration over chains.

Denote by $\mathcal{CM}(\Gamma\backslash \Omega )$ the space of conformal metrics on $\Gamma\backslash\Omega$. That is, every ${\rm d}s^2\in \mathcal{CM}(\Gamma\backslash \Omega)$ is represented as ${\rm d}s^2={\rm e}^{\phi(z)}|{\rm d}z|^2$, where $\phi$ is a smooth function on $\Omega$ satisfying
\begin{gather}\label{e:phi-behave}
\phi\circ\gamma+\log|\gamma'|^2=\phi\quad\quad\forall \;\gamma\in\Gamma.
\end{gather}

 The Liouville action is a function on the space of conformal metrics. Its construction is as follows.
Starting with the 2-form
\begin{gather*} 
\omega[\phi] = \big(|\phi_{z}|^2 +{\rm e}^{\phi}\big){\rm d}z\wedge {\rm d}\bar{z}\in \mathsf{C}^{2,0},
\end{gather*}
we have
\begin{gather*}
\delta\omega[\phi]={\rm d}\check{\theta}[\phi],
\end{gather*}
where $\check{\theta}[\phi] \in \mathsf{C}^{1,1}$ is given explicitly by
\begin{gather*} 
\check{\theta}_{\gamma^{-1}}[\phi]
= \bigg(\phi -\frac{1}{2}\log|\gamma'|^2-2\log 2-\log|c(\gamma)|^2\bigg)
\bigg(\frac{\gamma''}{\gamma'} {\rm d}z -\frac{\ov{\gamma''}}{\ov{\gamma'}}{\rm d}\bar{z}\bigg).
\end{gather*}
Here $c(\gamma)$ is the element $c$ in the linear fractional transformation $\di\gamma=\left(\begin{smallmatrix} a & b\\c & d\end{smallmatrix}\right)\!$. Notice that \mbox{$\check{\theta}_{\gamma^{-1}}[\phi]=0$} if $c(\gamma)=0$.

Next, set
\begin{gather*}
\check{u}=\delta \check{\theta}[\phi] \in \mathsf{C}^{1,2}.
\end{gather*}
From the definition of $\check{\theta}$ and $\delta^2=0$, it follows that the
1-form $\check{u}$ is closed. An explicit calculation gives
\begin{align*} \check{u}_{\gamma_1^{-1},\gamma_2^{-1}}= & -\bigg(\frac{1}{2}\log|\gamma_1'|^2+\log\frac{|c(\gamma_2)|^2}{|c(\gamma_2\gamma_1)|^2}\bigg)
\bigg(\frac{\gamma_2''}{\gamma_2'}\circ\gamma_1\, \gamma_1'\, {\rm d}z -
\frac{\ov{\gamma_2''}}{\ov{\gamma_2'}}\circ\gamma_1\, \ov{\gamma_1'}\,{\rm d}\bar{z}\bigg)
\\
 & + \bigg(\frac{1}{2}\log|\gamma_2'\circ\gamma_1|^2+\log\frac{|c(\gamma_2\gamma_1)|^2}{|c(\gamma_1)|^2}\bigg) \bigg(\frac{\gamma_1''}{\gamma_1'} {\rm d}z -
\frac{\ov{\gamma_1''}}{\ov{\gamma_1'}}{\rm d}\bar{z}\bigg).
\end{align*}

For ${\rm e}^{\phi(z)}|{\rm d}z|^2\in \mathcal{CM}(\Gamma\backslash \Omega)$, the Liouville action is defined as
\begin{gather*}
S[\phi]=\frac{\rm i}{2}\bigl(\langle \omega[\phi], F_1-F_2 \rangle-\big\langle\check{\theta}[\phi], L_1-L_2\big\rangle+\big\langle \check{u}, W_1-W_2\big\rangle\bigr).
\end{gather*}
Here $W_1=J_1(W)$ and $W_2=J_1(W')$ for the chain $W'$ in $\mathbb{L}$ given by the complex conjugation of~$W$. We~also define
\begin{gather}\label{e:def-check-S}
\check{S}[\phi]=\frac{\rm i}{2}\bigl(\big\langle \check{\omega}[\phi], F_1-F_2\big\rangle-\big\langle\check{\theta}[\phi], L_1-L_2\big\rangle+\langle \check{u}, W_1-W_2\rangle\bigr),
\end{gather}
where
\begin{gather*}
\check{\omega}[\phi] = |\phi_{z}|^2 {\rm d}z\wedge {\rm d}\bar{z}\in \mathsf{C}^{2,0}.
\end{gather*}
Note that $\delta \omega[\phi] =\delta \check{\omega}[\phi]$ since ${\rm e}^\phi \,{\rm d}z\wedge {\rm d}\bar{z}$ is $\Gamma$-invariant.

\section{Liouville action for harmonic diffeomorphisms}\label{s:harmonic}

For compact Riemann surfaces $X$ and $Y$ with genus $g\geq 2$, let $h\colon X\to Y$ denote a harmonic map for the hyperbolic metric ${\rm e}^{\psi(u)} |{\rm d}u|^2$ on $Y$, where $u$ denotes a conformal coordinate on $Y$.
The harmonic condition for $h$ is given by
\begin{gather}\label{e:harmonic-condition}
h_{z\bar{z}}+\big(\psi_u\circ h \big) h_z h_{\bar{z}}=0
\end{gather}
for a conformal coordinate $z$ on $X$.
Note that this condition depends on the conformal structure on $X$ and the metric structure on $Y$.
The pullback metric $h^*\big( {\rm e}^{\psi(u)} |{\rm d}u|^2\big)$ on $X$ has the following expression
\begin{gather}\label{e:pullback-metric}
h^*\big( {\rm e}^{\psi(u)} |{\rm d}u|^2\big) = {\rm e}^{\psi\circ h} h_z \bar{h}_z {\rm d}z^2 + {\rm e}^{\psi\circ h} h_z \bar{h}_{\bar{z}} |{\rm d}z|^2 + {\rm e}^{\psi\circ h} h_{\bar{z}} \bar{h}_z |{\rm d}z|^2 + {\rm e}^{\psi\circ h} h_{\bar{z}} \bar{h}_{\bar{z}} {\rm d}\bar{z}^2.
\end{gather}

We denote the $(2,0)$-component of $h^*\big( {\rm e}^{\psi(u)} |{\rm d}u|^2\big)$ by
\begin{gather*}
\Phi(h)={\rm e}^{\psi\circ h} h_z \bar{h}_z,
\end{gather*}
which is called a \emph{Hopf differential} of $h$. By the harmonicity condition \eqref{e:harmonic-condition}, one can easily check that $\Phi(h)$ is a holomorphic quadratic differential on $X$. We~also put
\begin{gather*}
{\rm e}^{\phi}:= {\rm e}^{\psi\circ h} h_z \bar{h}_{\bar{z}}.
\end{gather*}
By~\cite[Theorem 3.10.1 and Corollary 3.10.1]{J}, we have

\begin{Proposition}\label{p:property-har}
For a compact Riemann surface $X$ and a compact hyperbolic Riemann surface~$Y$ with same genus $g\ge 2$ and a continuous map $g\colon X\to Y$ of degree $1$,
there is a unique harmonic diffeomorphism $h\colon X\to Y$ in the homotopy class of $g$. In~this case,
${\rm e}^{\phi}= {\rm e}^{\psi\circ h} h_z \bar{h}_{\bar{z}}$ never vanishes on $X$, where ${\rm e}^{\psi(u)} |{\rm d}u|^2$ is the hyperbolic metric on $Y$.
\end{Proposition}

By Proposition \ref{p:property-har}, for a harmonic diffeomorphism $h\colon X\to Y$ for a hyperbolic metric on $Y$, ${\rm e}^\phi={\rm e}^{\psi\circ h} h_z \bar{h}_{\bar{z}}$ defines a metric on $X$.
Similarly $|\Phi(h)|$ defines a singular flat metric on $X$ since the holomorphic quadratic differential $\Phi(h)$ should have $-2\chi(X)$ number of zeros.

\begin{Proposition}For a harmonic diffeomorphism $h\colon X\to Y$ for a hyperbolic metric on $Y$,
\begin{gather}\label{e:Gauss-cur}
K_{\phi}:=-2\phi_{z\bar{z}} {\rm e}^{-\phi}= \frac{|h_{\bar{z}}|^2}{|h_z|^2}-1.
\end{gather}
In particular, $K_{\phi}$ never vanishes.
\end{Proposition}

\begin{proof}By the equality \eqref{e:harmonic-condition},
\begin{gather}\label{e:phi-z}
\phi_z= (\psi_u\circ h)\,\displaystyle h_z +\frac{h_{zz}}{h_z},
\end{gather}
and
\begin{gather}\label{e:phi-z-bar-z}
\phi_{z\bar{z}}= (\psi_{u\bar{u}}\circ h)\, \big( |h_z|^2 -|h_{\bar{z}}|^2 \big) = \frac12 {\rm e}^{\psi\circ h} \big( |h_z|^2 -|h_{\bar{z}}|^2 \big).
\end{gather}
Here the second equality in \eqref{e:phi-z-bar-z} follows by the Liouville equation for ${\rm e}^{\psi}$,
\begin{gather}\label{e:liouv}
\psi_{u\bar{u}}=\frac12 {\rm e}^{\psi}.
\end{gather}
Hence we have
\begin{gather*}
K_{\phi}=-2\phi_{z\bar{z}} {\rm e}^{-\phi} = -{\rm e}^{\psi\circ h} \big( |h_z|^2 -|h_{\bar{z}}|^2 \big)\cdot {\rm e}^{-\psi\circ h} |h_z|^{-2} = \frac{|h_{\bar{z}}|^2}{|h_z|^2}-1.
\end{gather*}
For the harmonic diffeomorphism $h\colon X\to Y$, its Jacobian $J_h= |h_z|^2 - |h_{\bar{z}}|^2$ never vanishes. Hence $K_\phi$ never vanishes
by the above equality.
\end{proof}

\begin{Proposition}\label{p:energy-momentum}
For a harmonic diffeomorphism $h\colon X\to Y$ for a hyperbolic metric on $Y$,
 ${\rm e}^\phi$~satisfies the following equality:
\begin{gather*}
\phi_{zz}-\frac12 \phi_z^2 = \bigg(\bigg(\psi_{uu}-\frac12 \psi^2_u \bigg)\circ h\bigg) h_z^2+\frac12 \Phi(h) + \mathcal{S}(h) \qquad \text{on} \quad U,
\end{gather*}
 where $z$ is a conformal coordinate on an open set $U\subset X$ and $\mathcal{S}(h)= \frac{h_{zzz}}{h_z}-\frac 32 \big(\frac{h_{zz}}{h_z}\big)^2$.
\end{Proposition}

\begin{proof}From \eqref{e:phi-z}, we have
\begin{gather}\label{e:phi-zz}
\phi_{zz}=(\psi_{uu}\circ h) h_z^2 +(\psi_{u\bar{u}}\circ h) h_z \bar{h}_z +(\psi_u\circ h) h_{zz} +\frac{h_{zzz}}{h_z} -\bigg(\frac{h_{zz}}{h_z}\bigg)^2.
\end{gather}
Then the claimed equality follows by \eqref{e:phi-z}, \eqref{e:liouv}, and \eqref{e:phi-zz}.
\end{proof}

For two marked compact Riemann surfaces of genus $g\ge2$, there exists a marked, normalized quasi-Fuchsian group $\Gamma$ such that $X=\Gamma\backslash \Omega_1$ and $\overline{Y}=\Gamma\backslash \Omega_2$ for a region of discontinuity
$\Omega_1\sqcup \Omega_2$ by Bers' simultaneous uniformization theorem. We~also assume that
$Y$ is realized by the Fuchsian uniformization with a Fuchsian group $\Gamma_Y$ acting on the upper half plane $\mathbb{U}$ such that $Y = \Gamma_Y\backslash \mathbb{U}$. By Proposition \ref{p:property-har}, for these marked compact Riemann surfaces $X$ and $Y$, there exists the unique harmonic diffeomorphism $h\colon X\to Y$ for the hyperbolic metric ${\rm e}^{\psi(u)}|{\rm d}u|^2$ on $Y$ such that $h$ maps the marking of $X$ to the marking of $Y$. This also induces a harmonic map from $\Omega_1$ to $\mathbb{U}$, denoted by the same notation $h$, such that for a given $\gamma\in\Gamma$, there is a~$\gamma_Y\in \Gamma_Y$ with
\begin{gather*}
h\circ \gamma =\gamma_Y \circ h.
\end{gather*}
Now we define a metric ${\rm e}^{\phi(z)}|{\rm d}z|^2$ on $\Omega_1\sqcup \Omega_2$ by the pullback of the hyperbolic metric ${\rm e}^{\psi(u)}|{\rm d}u|^2$ on $\mathbb{U}\sqcup \mathbb{L}$ by $h\colon \Omega_1\to\mathbb{U}$
and $J^{-1}_2\colon \Omega_2\to \mathbb{L}$ respectively.
More precisely we have
\begin{gather*}
{\rm e}^{\phi(z)}=\begin{cases}
{\rm e}^{\psi\circ h(z)} |h_z(z)|^2 &\text{for} \ z\in \Omega_1,
\\[.5ex]
 {\rm e}^{\psi\circ J_2^{-1}(z)} \big|\big(J_2^{-1}\big)_z(z)\big|^2 &\text{for} \ z\in \Omega_2.
 \end{cases}
\end{gather*}
Note that we take only the second part of the pullback metric given in \eqref{e:pullback-metric} for $z\in \Omega_1$ in the above definition of ${\rm e}^{\phi(z)}$. By the definition, it follows that ${\rm e}^{\phi(\gamma(z))} |\gamma_z|^2 = {\rm e}^{\phi(z)}$ for any $\gamma\in\Gamma$ as in~\eqref{e:phi-behave}.

Now the Liouville action for the harmonic diffeomorphism $h$ is defined by
\begin{gather*}
S[h]=\frac{\rm i}{2}\bigl(\langle {\omega}[\phi], F_1-F_2 \rangle-\big\langle\check{\theta}[\phi], L_1-L_2 \big\rangle+\big\langle \check{u}, W_1-W_2 \big\rangle\bigr)
\end{gather*}
and its modification is defined by
\begin{gather*}
\check{S}[h]=\frac{\rm i}{2}\big(\big\langle \check{\omega}[\phi], F_1-F_2 \big\rangle-\big\langle\check{\theta}[\phi], L_1-L_2 \big\rangle+\big\langle \check{u}, W_1-W_2 \big\rangle\big).
\end{gather*}
The holomorphic energy of $h$ is defined by
\begin{gather*}
E(h) = \int_{\Gamma\backslash \Omega_1} {\rm e}^{\phi} {\rm d}^2z = \int_{\Gamma\backslash\Omega_1} {\rm e}^{\psi\circ h} h_z \bar{h}_{\bar{z}}\, {\rm d}^2z.
\end{gather*}
From the definitions we have
\begin{gather*}
S[h]=\check{S}[h]+E(h)+ 2\pi (2g-2).
\end{gather*}

\begin{Remark}\label{r:def-with-Hopf}
In the above definition of the Liouville action for diffeomorphisms $S[h]$, one may use $|\Phi(h)|=\big|{\rm e}^{\psi\circ h(z)} h_z \bar{h}_z\big|$ instead of
${\rm e}^\phi={\rm e}^{\psi \circ h(z)} h_z \bar{h}_{\bar{z}}$. Since $|\Phi(h)|$ defines a singular flat metric on $X$, the corresponding term $\check{\omega}[\phi]$ defined by $|\Phi(h)|$
is singular where $|\Phi(h)|$ has a zero.
To~deal with these singularities, we need to regularize the integral $\langle \check{\omega}[\phi], F_1-F_2 \rangle$ at the singular points as in \cite{M-P}.
\end{Remark}

Given a harmonic Beltrami differential $\mu\in \mathcal{B}^{-1,1}(\Gamma)$,
let $f^\vep=f^{\vep\mu}$ be the unique quasi-conformal map satisfying \eqref{e:beltrami-eq} with the Beltrami differential $\vep\mu$.
Notice that $f^{\vep}$ varies holomorphically with respect to $\vep$ and thus
\begin{gather*}
\frac{\pa}{\pa\bar{\vep}} \bigg|_{\vep=0}f^{\vep}=0.
\end{gather*}
Let
\begin{gather*}
\dot{f}= \frac{\pa}{\pa\vep}\bigg|_{\vep=0}f^{\vep}.
\end{gather*}
It follows from the definition $f^{\vep}_{\bar{z}}=\vep\mu f^{\vep}_z$ that
\begin{gather*}
\dot{f}_{\bar{z}}=\mu.
\end{gather*}
For any linear fractional transformation $\gamma$, let
\begin{gather*}
\gamma^{\vep\mu}=f^{\vep\mu}\circ\gamma\circ (f^{\vep\mu})^{-1}.
\end{gather*}
Then $\gamma^{\vep\mu}$ varies holomorphically with respect to $\vep$.
The Lie derivative of the smooth family of~$(l, m)$ tensors $\omega^{\vep\mu}$ on $\mathfrak{D}(\Gamma^{\vep\mu})$ for $\Gamma^{\vep\mu}= f^{\vep\mu}\circ\Gamma\circ (f^{\vep\mu})^{-1}$ is defined as
\begin{gather*}
L_{\mu}\omega=\frac{\pa}{\pa\vep}\bigg|_{\vep=0}\omega^{\vep\mu}\circ f^{\vep\mu} (f^{\vep\mu}_z)^l\big(\overline{f^{\vep\mu}_z}\big)^m.
\end{gather*}

For a harmonic diffeomorphism $h\colon X\to Y$, we consider the situation of varying harmonic diffeomorphisms along a variation of $X$ with a fixed $Y$.
For this purpose, we put $\mathcal{B}^{-1,1}(\Gamma,\Omega_1)$ to be the subspace of
$\mathcal{B}^{-1,1}(\Gamma)$ consisting of $\mu\in \mathcal{B}^{-1,1}(\Gamma)$ whose support lies in $\Omega_1$. Hence, a Beltrami differential $\mu\in\mathcal{B}^{-1,1}(\Gamma,\Omega_1)$ represents an element in $\mathfrak{D}(\Gamma,\Omega_1)$.
Now we have the following commuting diagram with $f^\vep=f^{\vep\mu}$ for $\mu\in\mathcal{B}^{-1,1}(\Gamma,\Omega_1)$,
\begin{gather*}
\begin{CD}
 X @>f^\vep >> X^\vep \\
 @VV h V @VV h^\vep V \\
 Y @> g^\vep>> Y,
\end{CD}
\end{gather*}
where $X^\vep= f^\vep(X)$ and $g^\vep:= h^\vep\circ f^\vep\circ h^{-1}\colon Y\to Y$.

For the Lie derivative $L_\mu S[h]=L_\mu \check{S}[h]+ L_\mu E(h)$, first we consider $L_\mu E(h)$.

\begin{Theorem}\label{t:E-variation}
For a harmonic Beltrami differential $\mu\in \mathcal{B}^{-1,1}(\Gamma,\Omega_1)$,
\begin{gather*}
L_{\mu} E(h) =- \int_{\Gamma\backslash \Omega_1} \Phi(h) \mu\, {\rm d}^2z.
\end{gather*}
\end{Theorem}

\begin{proof}
By the definition of the Lie derivative $L_{\mu}$, we have the following equalities.
\begin{gather*}
\frac{\partial}{\partial\vep} \bigg|_{\vep=0} \big({\rm e}^{\psi\circ h^\vep\circ f^\vep} (h^\vep_z\circ f^\vep) \big( \bar{h}^\vep_{\bar{z}}\circ f^\vep \big) \, {\rm d}f^\vep\wedge {\rm d}\bar{f}^\vep \big)
\\ \qquad
{}= {\rm e}^{\psi\circ h} \big((\psi_u\circ h) \big(\dot{h}+h_z \dot{f} \big) +(\psi_{\bar{u}}\circ h) \big(\dot{\bar{h}} +\bar{h}_z \dot{f} \big)\big)h_z \bar{h}_{\bar{z}}\, {\rm d}z\wedge {\rm d}\bar{z}
\\ \qquad \phantom{=}
{}+ {\rm e}^{\psi\circ h} \big( \big(\dot{h}_z +h_{zz} \dot{f}\big) \bar{h}_{\bar{z}} + h_z \big(\dot{\bar{h}}_{\bar{z}} +\bar{h}_{\bar{z}z} \dot{f}\big)\big)\, {\rm d}z\wedge {\rm d}\bar{z}
+ {\rm e}^{\psi\circ h} h_z \bar{h}_{\bar{z}} \dot{f}_z \, {\rm d}z\wedge {\rm d}\bar{z}.
\end{gather*}
Hence,
\begin{align*}
L_\mu E(h) ={}& \int_ {\Gamma\backslash \Omega_1} {\rm e}^{\psi\circ h} \big((\psi_u\circ h) h_z \bar{h}_{\bar{z}}\big(\dot{h}+h_z \dot{f} \big) +\big(\dot{h} +h_{z} \dot{f}\big)_z \bar{h}_{\bar{z}} \big)\, {\rm d}^2z
\\
&+ \int_{\Gamma\backslash \Omega_1} {\rm e}^{\psi\circ h} \big((\psi_{\bar{u}}\circ h) h_z \bar{h}_{\bar{z}}\big(\dot{\bar{h}}+\bar{h}_z \dot{f} \big) +h_z\big(\dot{\bar{h}} +\bar{h}_{z} \dot{f}\big)_{\bar{z}} \big)\, {\rm d}^2z
\\
&-\int_{\Gamma\backslash \Omega_1} {\rm e}^{\psi\circ h} h_z \bar{h}_{{z}} \dot{f}_{\bar{z}} \, {\rm d}^2z.
\end{align*}
By integration by parts, we have
\begin{align*}
L_\mu E(h)={}& \int_{\Gamma\backslash \Omega_1} {\rm e}^{\psi\circ h} \big( (\psi_u\circ h)h_z \bar{h}_{\bar{z}} \big(\dot{h}+h_z \dot{f} \big) -(\psi_u\circ h)h_z \bar{h}_{\bar{z}} \big(\dot{h}+h_z \dot{f} \big) \big)\, {\rm d}^2z
\\
&-\int_{\Gamma\backslash \Omega_1}{\rm e}^{\psi\circ h} \big((\psi_{\bar{u}}\circ h) \bar{h}_z \bar{h}_{\bar{z}} \big(\dot{h}+h_z \dot{f} \big) + \bar{h}_{z\bar{z}} \big(\dot{h} +h_{z} \dot{f}\big)\big)\, {\rm d}^2z
\\
&+ \int_{\Gamma\backslash \Omega_1} {\rm e}^{\psi\circ h} \big( (\psi_{\bar{u}}\circ h)h_z \bar{h}_{\bar{z}} \big(\dot{\bar{h}}+\bar{h}_z \dot{f} \big) - (\psi_{\bar{u}}\circ h)h_z \bar{h}_{\bar{z}} \big(\dot{\bar{h}}+\bar{h}_z \dot{f} \big) \big) {\rm d}^2z
\\
&-\int_{\Gamma\backslash \Omega_1} {\rm e}^{\psi\circ h} \big( \big(\psi_{{u}}\circ h\big)h_z {h}_{\bar{z}} \big(\dot{\bar{h}}+\bar{h}_z \dot{f} \big)
+h_{z\bar{z}}\big(\dot{\bar{h}} +\bar{h}_{z} \dot{f}\big) \big)\, {\rm d}^2z
\\
&-\int_{\Gamma\backslash \Omega_1} {\rm e}^{\psi\circ h} h_z \bar{h}_{{z}} \dot{f}_{\bar{z}} \, {\rm d}^2z
=- \int_{\Gamma\backslash \Omega_1} {\rm e}^{\psi\circ h} h_z \bar{h}_{\bar{z}}\dot{f}_{\bar{z}} \, {\rm d}^2z.
\end{align*}
Here the last equality holds by the equality \eqref{e:harmonic-condition}.
\end{proof}

Let us remark that the equation \eqref{e:harmonic-condition} for the harmonic map condition is the Euler--Lagrange equation for the holomorphic energy functional given by
\begin{gather*}
E(g^\vep\circ h)=E(h^\vep\circ f^\vep) = \int_{\Gamma\backslash \Omega_1} {\rm e}^{\psi\circ h^\vep \circ f^\vep} |(h^\vep \circ f^\vep)_z|^2 \ {\rm d}^2z.
\end{gather*}
This can be checked easily as in the proof of Theorem \ref{t:E-variation}.
This will be also used crucially in the proof of the following theorem.

\begin{Theorem}\label{t:final-thm}
For a harmonic Beltrami differential $\mu\in\mathcal{B}^{-1,1}(\Gamma,\Omega_1)$,
\begin{gather}\label{e:variation-S-h}
L_\mu S[h] = \int_{\Gamma\backslash \Omega_1} \, \big(2\mathcal{S}(h) - K_\phi \Phi(h) \big)\, \mu \, {\rm d}^2z.
\end{gather}
\end{Theorem}

\begin{proof}
Putting $\lambda:=\dot{\phi}+\phi_z \dot{f} +\dot{f}_z$ with $\dot{\phi}=\frac{\rm d}{{\rm d}\vep}\big|_{\vep=0} \phi^{\vep\mu}$,
by Theorems \ref{t:E-variation} and~\ref{t:firstvariation},
\begin{gather*}
L_{\mu} S[h]= L_{\mu} \check{S}[h] +L_{\mu} E(h)
= \int_{\Gamma\backslash\Omega} \big(\big(2\phi_{zz}-\phi^2_z\big)\mu -2\phi_{z\bar{z}} \lambda\big)\, {\rm d}^2z -\int_{\Gamma\backslash\Omega_1} \Phi(h) \mu\, {\rm d}^2z.
\end{gather*}
For $\mu\in \mathcal{B}^{-1,1}(\Gamma,\Omega_1)$, $\mu$ vanishes over $\Omega_2$. Moreover the term $\lambda$ vanishes over $\Omega_2$ since the metric ${\rm e}^{\phi(z)}|{\rm d}z|^2$ is hyperbolic on $\Omega_2$. Hence, by Proposition \ref{p:energy-momentum} we have
\begin{align*}
L_{\mu} S[h]={}& \int_{\Gamma\backslash\Omega_1} \big((2\phi_{zz}-\phi^2_z)\mu -2\phi_{z\bar{z}} \lambda \big)\, {\rm d}^2z -\int_{\Gamma\backslash\Omega_1} \Phi(h) \mu\, {\rm d}^2z
\\
= {}&\int_{\Gamma\backslash\Omega_1} \bigg(\bigg(2\bigg(\bigg(\psi_{uu}-\frac12 \psi^2_u \bigg)\circ h\bigg) h_z^2+ \Phi(h) + 2\mathcal{S}(h) \bigg) \mu - 2\phi_{z\bar{z}}\lambda \bigg) \, {\rm d}^2z
\\
&-\int_{\Gamma\backslash\Omega_1} \Phi(h) \mu\, {\rm d}^2z
= \int_{\Gamma\backslash\Omega_1} \big( 2\mathcal{S}(h)\mu - 2\phi_{z\bar{z}}\lambda \big) \, {\rm d}^2z.
\end{align*}
Here, for the last equality, we used the fact that $\psi_{uu}-\frac12 \psi^2_u \equiv 0$ on $\mathbb{U}$ for ${\rm e}^\psi= (\mathrm{Im}(u))^{-2}$, where~$u$ denotes the global coordinate on $\mathbb{U}$.
Now we analyze the term $\lambda=\dot{\phi} +\phi_z \dot{f} + \dot{f}_z$ as follows. First, by definition,
\begin{gather}\label{e:lambda-var}
\dot{\phi}+\phi_z\dot{f}+\dot{f}_z = \psi_u\circ h \big(\dot{h}\!+h_z \dot{f} \big) +\psi_{\bar{u}} \circ h \big(\dot{\bar{h}}\!+ \bar{h}_z \dot{f}\big)
 + \frac{\dot{h}_z\!+h_{zz}\dot{f} \!+h_z \dot{f}_z}{ h_z}+ \frac{\dot{\bar{h}}_{\bar{z}}\!+\bar{h}_{\bar{z}z}\dot{f}}{ \bar{h}_{\bar{z}}}.
\end{gather}
On the other hand, recalling that the harmonic diffeomorphism $h\colon X\to Y$ is a critical point of the holomorphic energy functional along the variation $g^\vep \circ h$,
\begin{gather*}
0=\frac{\partial}{\partial\vep}\bigg|_{\vep=0} \big({\rm e}^{\psi\circ g^\vep \circ h} |(g^\vep \circ h)_z|^2 \big)
=\frac{\partial}{\partial\vep}\bigg|_{\vep=0} \big({\rm e}^{\psi\circ h^\vep \circ f^\vep} |(h^\vep \circ f^\vep)_z|^2 \big),
\end{gather*}
so that
\begin{gather}
0= {\rm e}^{\psi\circ h} |h_z|^2\bigg(\psi_u\circ h \big(\dot{h}+h_z \dot{f} \big) +\psi_{\bar{u}} \circ h \big( \dot{\bar{h}}+ \bar{h}_z \dot{f} \big)\nonumber
\\ \hphantom{0= {\rm e}^{\psi\circ h} |h_z|^2\bigg(}
{} + \frac{\dot{h}_z+h_{zz}\dot{f} +h_z \dot{f}_z}{ h_z}+ \frac{\dot{\bar{h}}_{\bar{z}}+\bar{h}_{\bar{z}z}\dot{f}+\bar{h}_{z}\dot{f}_{\bar{z}} }{ \bar{h}_{\bar{z}}} \bigg).\label{e:Ahlfors}
\end{gather}
Hence, by \eqref{e:lambda-var} and \eqref{e:Ahlfors} we have
\begin{gather}\label{e:ahlfors-app}
\dot{\phi}+\phi_z\dot{f}+\dot{f}_z = -\frac{\bar{h}_{z}\dot{f}_{\bar{z}} }{ \bar{h}_{\bar{z}}}.
\end{gather}
Finally, by \eqref{e:Gauss-cur}, \eqref{e:phi-z-bar-z}, and \eqref{e:ahlfors-app},
\begin{gather*}
- 2\phi_{z\bar{z}} \lambda= {\rm e}^{\psi\circ h} \big(|h_{{z}}|^2- |h_{\bar{z}}|^2\big) \frac{\bar{h}_{z} }{ \bar{h}_{\bar{z}}} \mu
= {\rm e}^{\psi\circ h} \bigg( h_{{z}} \bar{h}_{z}- h_{{z}} \bar{h}_{z} \frac{|h_{\bar{z}}|^2}{{|h_{z}|^2}} \bigg) \mu = -K_{\phi} \Phi(h) \mu.
\end{gather*}
This completes the proof.
\end{proof}

\begin{Remark}
When $X$ and $Y$ are the same Riemann surface, the harmonic diffeomorphism $h\colon X\to Y$ is induced by $J_1^{-1}\colon \Omega_1\to \mathbb{U}$ so that its Hopf differential $\Phi(h)$ vanishes. Hence, the variation formula \eqref{e:variation-S-h} simplifies at the origin point $X=Y$ in $\mathfrak{D}(\Gamma,\Omega_1) \simeq \mathfrak{T}(\Gamma_1)$.
This may suggest that the second variation formula for $S[h]$ would be simpler at the origin $X=Y$ than other points in $\mathfrak{D}(\Gamma,\Omega_1)$. This is the case of
the energy functional of harmonic diffeomorphisms whose second variation gives the Weil--Petersson symplectic 2-form at $X=Y$ (see~\cite[Corollary~5.8]{W89} and~\cite[Theorem~3.1.3]{Tromba}).
\end{Remark}

\section{Variation of Liouville action}\label{s:variation}

In this section, we compute the variation of the Liouville action defined for any smooth conformal metric. Most of the computations are similar to the one given in \cite{2}, where a smooth family of conformal metrics is given by the hyperbolic metrics.
However, we will have some additional terms since we do not assume the hyperbolic metric condition.
On the other hand, we will also see that the variational argument developed in \cite{2} works well for a smooth family of conformal metrics and these additional terms can be nicely organized.

Now we decompose the Liouville action $S=S[\phi]$ into two parts by
\begin{gather*}
S[\phi]= \check{S}[\phi] + \int_{\Gamma\backslash \Omega} {\rm e}^{\phi} {\rm d}^2z,
\end{gather*}
where $\check{S}[\phi]$ is defined in \eqref{e:def-check-S}. First we deal with the variation of $\check{S}=\check{S}[\phi]$.

For a harmonic Beltrami differential $\mu\in\mathcal{B}^{-1,1}(\Gamma)$, let $f^{\vep}=f^{\vep\mu}\colon X\to X^{\vep}$ denote the quasi-conformal map satisfying the Beltrami equation \eqref{e:beltrami-eq}.

\begin{Theorem}\label{t:firstvariation}
For a smooth family of conformal metrics ${\rm e}^{\phi^{\vep\mu}(z^\vep)}|{\rm d}z^\vep|^2$ on $X^\vep$,
\begin{gather*}
L_{\mu} \check{S}[\phi]=\int_{\Gamma\backslash \Omega} \big(\big(2\phi_{zz}-\phi^2_z\big)\mu -2\phi_{z\bar{z}} \lambda\big)\, {\rm d}^2z,
\end{gather*}
where $\lambda=\dot{\phi}+\phi_z \dot{f} +\dot{f}_z$ with $\dot{\phi}=\frac{\rm d}{{\rm d}\vep}|_{\vep=0} \phi^{\vep\mu}$.
\end{Theorem}

Most of the remaining part of this section is a proof of Theorem \ref{t:firstvariation}. By definition,
\begin{gather}\label{e:def-var}
L_{\mu}\check{S}[\phi] =\frac{\rm i}{2}\bigl(\langle L_{\mu}\check{\omega}, F_1-F_2 \rangle-\big\langle L_{\mu}\check{\theta}, L_1-L_2 \big\rangle+\langle L_{\mu}\check{u}, W_1-W_2 \rangle\bigr).
\end{gather}
To deal with the first term on the right hand side of \eqref{e:def-var}, we start with some lemmas.
\begin{Lemma}\label{l:1}
The following equality holds
\begin{gather}\label{e:1}
L_{\mu}\check{\omega}
=\big(\big(2\phi_{zz}-\phi^2_z\big)\mu -2\phi_{z\bar{z}} \lambda \big)\, {\rm d}z\wedge {\rm d}\bar{z} -{\rm d}(\phi_z \lambda \,{\rm d}z) +{\rm d}(\phi_{\bar{z}} \lambda\, {\rm d}\bar{z}) -{\rm d}\xi,
\end{gather}
where $\lambda=\dot{\phi}+\phi_z \dot{f} +\dot{f}_z$,\, $\xi=2\phi_z\dot{f}_{\bar{z}}{\rm d}\bar{z}-\phi \,{\rm d}\dot{f}_z$.
\end{Lemma}

\begin{proof}
The proof is just a straightforward computation as follows.
\begin{align*}
L_{\mu}\check{\omega}={}& \frac{\pa}{\pa\vep}\bigg|_{\vep=0} \big((\phi^{\vep\mu})_z\circ f^{\vep\mu} {\rm d}f^{\vep\mu} \wedge(\phi^{\vep\mu})_{\bar{z}}\circ f^{\vep\mu} {\rm d}\bar{f}^{\vep\mu} \big)
\\[.3ex]
={}&\big(\dot{\phi}_z +\phi_{zz} \dot{f} +\phi_z \dot{f}_{{z}}\big)\phi_{\bar{z}}\, {\rm d}z\wedge {\rm d}\bar{z}+ \phi_z\big(\dot{\phi}_{\bar{z}}+\phi_{z\bar{z}}\dot{f}\big)\, {\rm d}z\wedge {\rm d}\bar{z}
\\[.3ex]
={}&\big(\dot{\phi} +\phi_{z} \dot{f} + \dot{f}_{{z}}\big)_z\phi_{\bar{z}}\, {\rm d}z\wedge {\rm d}\bar{z}+ \phi_z\big(\dot{\phi}+\phi_{z}\dot{f}+\dot{f}_z\big)_{\bar{z}} \, {\rm d}z\wedge {\rm d}\bar{z}
\\[.3ex]
 &-\big(\phi_{\bar{z}} \dot{f}_{zz} +\phi_z\big(\phi_z\dot{f}_{\bar{z}}+\dot{f}_{z\bar{z}}\big)\big)\, {\rm d}z\wedge {\rm d}\bar{z}
 \\[.3ex]
 ={}&\big(2\phi_{zz}-\phi^2_z\big)\mu\, {\rm d}z\wedge {\rm d}\bar{z}-{\rm d}\big(2\phi_z\dot{f}_{\bar{z}}{\rm d}\bar{z}-\phi \,{\rm d}\dot{f}_z\big) -{\rm d}(\phi_z \lambda \,{\rm d}z) +{\rm d}(\phi_{\bar{z}} \lambda \,{\rm d}\bar{z})
 \\[.3ex]
 &-2\phi_{z\bar{z}} \lambda\, {\rm d}z\wedge {\rm d}\bar{z}.\tag*{\qed}
\end{align*}
\renewcommand{\qed}{}
\end{proof}

\begin{Remark}
The term $\lambda=\dot{\phi}+\phi_z \dot{f} +\dot{f}_z$ in Lemma \ref{l:1} vanishes when the metrics ${\rm e}^{\phi(z^\vep)} |{\rm d}z^\vep|^2$ are the hyperbolic metrics on $X^\vep$ by the work of Ahlfors in \cite{3}.
\end{Remark}

By Lemma \ref{l:1}, the equality \eqref{e:def-var} can be rewritten as follows:
\begin{gather}
L_{\mu}\check{S}[\phi]
 =\frac{\rm i}{2}\bigl(\langle L_{\mu}\check{\omega}, F_1-F_2 \rangle-\big\langle L_{\mu}\check{\theta}, L_1-L_2 \big\rangle+\langle L_{\mu}\check{u}, W_1-W_2 \rangle\bigr)\nonumber
 \\ \hphantom{L_{\mu}\check{S}[\phi]}
{}=\frac{\rm i}{2}\bigl(\big\langle \big(\big(2\phi_{zz}-\phi^2_z\big)\mu -2\phi_{z\bar{z}} \lambda \big)\, {\rm d}z\wedge {\rm d}\bar{z}, F_1-F_2 \big\rangle \nonumber
\\ \hphantom{L_{\mu}\check{S}[\phi]=\frac{\rm i}{2}\bigl(}
{} -\!\langle {\rm d}(\phi_z \lambda \,{\rm d}z)\! - {\rm d}(\phi_{\bar{z}} \lambda\, {\rm d}\bar{z})+ {\rm d}\xi, F_1\!-F_2 \rangle\! -\langle L_{\mu}\check{\theta}, L_1-L_2 \rangle +\langle L_{\mu}\check{u}, W_1\!-W_2 \rangle\bigr)\nonumber
\\ \hphantom{L_{\mu}\check{S}[\phi]}
{}=\frac{\rm i}{2}\bigl(\big\langle\big(\big(2\phi_{zz}\!-\phi^2_z\big)\mu \!-2\phi_{z\bar{z}} \lambda \big)\, {\rm d}z\wedge {\rm d}\bar{z}, F_1\!-F_2 \big\rangle \!-\langle \delta(\phi_z \lambda \,{\rm d}z\!- \phi_{\bar{z}} \lambda\, {\rm d}\bar{z} + \xi), L_1\!-L_2 \rangle\nonumber
\\ \hphantom{L_{\mu}\check{S}[\phi]=\frac{\rm i}{2}\bigl(}
{} -\langle L_{\mu}\check{\theta}, L_1-L_2 \rangle+\langle L_{\mu}\check{u}, W_1-W_2 \rangle\bigr),
\label{e:der-var1}
\end{gather}
where the third equality follows from $\partial' F_i=\partial'' L_i$ for $i=1,2$.
To deal with terms in the last line of \eqref{e:der-var1} together, let us put
\begin{gather}\label{e:def-chi}
{\chi}:= \delta\xi + \delta(\phi_z \lambda \,{\rm d}z - \phi_{\bar{z}} \lambda\, {\rm d}\bar{z}) + L_{\mu} \check{\theta}.
\end{gather}

First we have
\begin{Lemma}
The ${\chi}$ satisfies that $d\chi=0$ and $\delta\chi=L_\mu\check{u}$ on $\Omega$.
\end{Lemma}
\begin{proof}
The second equality follows easily by
\begin{gather*}
\delta\chi=\delta\big(\delta\xi + \delta(\phi_z \lambda \,{\rm d}z - \phi_{\bar{z}} \lambda\, {\rm d}\bar{z}) + L_{\mu} \check{\theta}\big) = \delta L_\mu \check{\theta}= L_\mu\delta\check{\theta}=L_\mu\check{u}.
\end{gather*}
To show the first equality $d\chi=0$, we start with some equalities. For the following equality
\begin{gather}\label{e:basic-eq-var}
\phi^{\vep\mu}\circ \gamma^{\vep\mu} +\log (\gamma^{\vep\mu})' +\log (\bar{\gamma}^{\vep\mu})' = \phi^{\vep\mu},
\end{gather}
we take derivative with respect to $\vep$ to obtain
\begin{gather}\label{e:app1}
\dot{\phi}\circ\gamma +\phi_z\circ \gamma \dot{\gamma} +\frac{\dot{\gamma}'}{\gamma'} =\dot{\phi}.
\end{gather}
We also take derivative with respect to $z$ and put $\vep=0$ for the equality \eqref{e:basic-eq-var} to get
\begin{gather}\label{e:app2}
\phi_z\circ\gamma \gamma' +\frac{\gamma''}{\gamma'} =\phi_z.
\end{gather}
Similarly taking derivative with respect to $\vep$ for $f^{\vep\mu}\circ \gamma=\gamma^{\vep\mu}\circ f^{\vep\mu}$, we have
\begin{gather}\label{e:app3}
\dot{f}\circ\gamma=\dot{\gamma}+ \gamma'\dot{f}.
\end{gather}
Using \eqref{e:app1}, \eqref{e:app2}, and \eqref{e:app3}, we observe that $\lambda$ satisfies
\begin{align*}
\lambda \circ \gamma &= \dot{\phi}\circ \gamma +\phi_z\circ \gamma \cdot \dot{f}\circ\gamma +\dot{f}_z\circ \gamma
\\
 &= \dot{\phi} -\bigg(\phi_z -\frac{\gamma''}{\gamma'}\bigg)\frac{\dot{\gamma}}{\gamma'} -\frac{\dot{\gamma}'}{\gamma'}
 + \frac{\phi_z-\frac{\gamma''}{\gamma'}}{\gamma'} (\dot{\gamma}+\gamma' \dot{f})
 + \bigg(\dot{f}_z+\frac{\dot{\gamma}'}{\gamma'} +\frac{\gamma''}{\gamma'} \dot{f}\bigg)
 \\
& = \dot{\phi}+\phi_z\dot{f} +\dot{f}_z =\lambda.
\end{align*}
Hence, $\lambda$ is $\Gamma$-invariant and this implies
\begin{gather}\label{e:lambda-gamma}
\big(\big(2\phi_{zz}-\phi^2_z\big)\mu -2\phi_{z\bar{z}} \lambda \big)\circ \gamma |\gamma'|^2\, {\rm d}z\wedge {\rm d}\bar{z} = \big(\big(2\phi_{zz}-\phi^2_z\big)\mu -2\phi_{z\bar{z}} \lambda \big)\, {\rm d}z\wedge {\rm d}\bar{z}.
\end{gather}
Recalling the definition of ${\chi}$ in \eqref{e:def-chi} and using the equalities \eqref{e:1} and \eqref{e:lambda-gamma},
\begin{align*}
{\rm d}\chi &= \delta\big({\rm d}\xi+ {\rm d} (\phi_z \lambda \,{\rm d}z- \phi_{\bar{z}} \lambda\, {\rm d}\bar{z})\big) + L_{\mu}{\rm d}\check{\theta}
\\
&= \delta\big({-}L_\mu \check{\omega}+ \big(\big(2\phi_{zz}-\phi^2_z\big)\mu -2\phi_{z\bar{z}} \lambda \big)\, {\rm d}z\wedge {\rm d}\bar{z}\big)+ L_{\mu}{\rm d}\check{\theta}
= - \delta L_\mu \check{\omega} +L_\mu \delta \check{\omega} =0.
\end{align*}
This completes the proof.
\end{proof}

\begin{Lemma}\label{l:l3}
\begin{gather*}
\delta\xi_{\gamma^{-1}}=-2\big(\dot{f}_{z\bar{z}}\circ\gamma\overline{\gamma'}-\dot{f}_{z\bar{z}}\big){\rm d}\bar{z}-\phi\, {\rm d}\big(\dot{f}_z\circ\gamma-\dot{f}_z\big)+\log|\gamma'|^2 {\rm d}\big(\dot{f}_z\circ\gamma\big).
\end{gather*}
\end{Lemma}

\begin{proof}
From the equality \eqref{e:app3},
\begin{gather}\label{e:dot-f-app}
\dot{f}_{z\bar{z}} \circ \gamma \gamma' \overline{(\gamma')} = \gamma''\dot{f}_{\bar{z}} +\gamma' \dot{f}_{z\bar{z}}.
\end{gather}
Then, from $\xi= 2\phi_z \dot{f}_{\bar{z}} {\rm d}\bar{z} -\phi \,{\rm d}\dot{f}_z$, using \eqref{e:dot-f-app} we have
\begin{align*}
\delta\xi_{\gamma^{-1}} &= \big(2\phi_z \dot{f}_{\bar{z}}\big) \circ \gamma \overline{\gamma'} {\rm d}\bar{z} -\big(\phi \,{\rm d}\dot{f}_z\big)\circ\gamma -2\phi_z\dot{f}_{\bar{z}}\, {\rm d}\bar{z} +\phi\, {\rm d}\dot{f}_z
\\[.5ex]
&=2\bigg( \phi_z -\frac{\gamma''}{\gamma'} \bigg)\dot{f}_{\bar{z}}\, {\rm d}\bar{z} -\big(\phi -\log|\gamma'|^2 \big) {\rm d}\dot{f}_z\circ\gamma - 2\phi_z\dot{f}_{\bar{z}}\, {\rm d}\bar{z} +\phi \,{\rm d}\dot{f}_z
\\[.5ex]
&=-2\big(\dot{f}_{z\bar{z}}\circ\gamma\overline{\gamma'}-\dot{f}_{z\bar{z}}\big){\rm d}\bar{z} -\phi \big( {\rm d}\dot{f}_{z}\circ\gamma - {\rm d}\dot{f}_z \big) + \log|\gamma'|^2 {\rm d}\dot{f}_z\circ\gamma.
\end{align*}
This completes the proof since $\big({\rm d}\dot{f}_z\big)\circ\gamma = {\rm d}\big( \dot{f}_z\circ\gamma\big)$.
\end{proof}

Now, for the term $L_{\mu}\check{\theta}$, we have

\begin{Lemma}\label{l:l4}
\begin{gather*}
L_{\mu}\check{\theta}_{\gamma^{-1}}
=\bigg(\phi-\frac{1}{2}\log|\gamma'|^2-2\log2-\log|c(\gamma)|^2\bigg){\rm d}
\big(\dot{f}_z\circ\gamma-\dot{f}_z\big)
\\ \hphantom{L_{\mu}\check{\theta}_{\gamma^{-1}}=}
{}-2\bigg(\dot{f}_{zz}\circ\gamma\gamma'-\dot{f}_{zz}-\lambda\frac{\gamma''}{\gamma'} \bigg){\rm d}z
+\bigg( \frac{1}{2}\big(\dot{f}_z\circ\gamma+\dot{f}_z\big)+\frac{\dot{c}(\gamma)}{c(\gamma)}-\lambda \bigg){\rm d}\log|\gamma'|^2.
\end{gather*}
\end{Lemma}

\begin{proof}
Recall
\begin{gather*} 
\check{\theta}_{\gamma^{-1}}[\phi]
= \bigg(\phi-\frac{1}{2}\log|\gamma'|^2-2\log2-\log|c(\gamma)|^2 \bigg)
 \bigg(\frac{\gamma''}{\gamma'} {\rm d}z -\frac{\ov{\gamma''}}{\ov{\gamma'}}{\rm d}\bar{z}\bigg).
\end{gather*}
For this, we observe the following equalities:
\begin{gather*}
\frac{\partial}{\partial\vep}\bigg|_{\vep=0}\big(\log\big|(\gamma^{\vep\mu})'\circ f^{\vep\mu}\big|^2 \big)= \frac{\dot{\gamma}'}{\gamma'}+\frac{\gamma''}{\gamma'}\dot{f} = \dot{f}_z\circ \gamma -\dot{f}_z,
\\[.5ex]
\frac{\partial}{\partial\vep}\bigg|_{\vep=0}
\bigg(\frac{(\gamma^{\vep\mu})''}{(\gamma^{\vep\mu})'}\circ f^{\vep\mu} \bigg)
=\dot{f}_{zz}\circ \gamma \gamma' -\dot{f}_{zz} -\frac{\gamma''}{\gamma'}\dot{f}_z.
\end{gather*}
For $c(\gamma)$, we also have
\begin{gather*}
-2c(\gamma)=\frac{\gamma''(z)}{(\gamma'(z))^{\frac{3}{2}}}.
\end{gather*}
Then
\begin{gather}
\frac{\dot{c}(\gamma)}{c(\gamma)}=\frac{\dot{f}_{zz}\circ\gamma\gamma'-\dot{f}_{zz}
-\frac{\gamma''}{\gamma'}\dot{f}_z}{\frac{\gamma''}{\gamma'}}-\frac{1}{2}
\big(\dot{f}_z\circ\gamma-\dot{f}_z\big)
=\frac{\dot{f}_{zz}\circ\gamma\gamma'-\dot{f}_{zz}}{\frac{\gamma''}{\gamma'}}-\frac{1}{2}
\big(\dot{f}_z\circ\gamma+\dot{f}_z\big).\label{e:cc}
\end{gather}
Hence,
\begin{gather*}
L_{\mu}\bigg(\phi-\frac{1}{2}\log|\gamma'|^2-2\log2-\log|c(\gamma)|^2\bigg)
\\ \qquad
{}=\lambda -\dot{f}_z-\frac{1}{2}\big(\dot{f}_z\circ\gamma-\dot{f}_z\big)
-\frac{\dot{c}(\gamma)}{c(\gamma)}
=\lambda -\frac{1}{2}\big(\dot{f}_z\circ\gamma+\dot{f}_z\big)-\frac{\dot{c}(\gamma)}{c(\gamma)}
=\lambda-\frac{\dot{f}_{zz}\circ\gamma\gamma'\!-\dot{f}_{zz}}{\frac{\gamma''}{\gamma'}}.
\end{gather*}
Moreover,
\begin{gather*}
L_{\mu}\bigg(\frac{\gamma''}{\gamma'}{\rm d}z-\frac{\overline{\gamma''}}{\overline{\gamma'}}{\rm d}\bar{z}\bigg)=L_{\mu}{\rm d}\log|\gamma'|^2={\rm d}\big(\dot{f}_z\circ\gamma-\dot{f}_z\big).
\end{gather*}
Hence,
\begin{align*}
L_{\mu}\check{\theta}_{\gamma^{-1}}
={}&\bigg(\phi-\frac{1}{2}\log|\gamma'|^2-2\log2-\log|c(\gamma)|^2\bigg)
{\rm d}\big(\dot{f}_z\circ\gamma-\dot{f}_z\big)
\\
&-\bigg( \frac{1}{2}\big(\dot{f}_z\circ\gamma+\dot{f}_z\big)
+\frac{\dot{c}(\gamma)}{c(\gamma)}-\lambda \bigg)\bigg(2\frac{\gamma''}{\gamma'}{\rm d}z-{\rm d}\log|\gamma'|^2\bigg)
\\
={}&\bigg(\phi-\frac{1}{2}\log|\gamma'|^2-2\log2-\log|c(\gamma)|^2\bigg)
{\rm d}\big(\dot{f}_z\circ\gamma-\dot{f}_z\big)
\\
&-2\bigg(\dot{f}_{zz}\circ\gamma\gamma'-\dot{f}_{zz}-\lambda\frac{\gamma''}{\gamma'} \bigg){\rm d}z
+\bigg( \frac{1}{2}\big(\dot{f}_z\circ\gamma+\dot{f}_z\big)
+\frac{\dot{c}(\gamma)}{c(\gamma)}-\lambda \bigg){\rm d}\log|\gamma'|^2.
\end{align*}
This completes the proof.
\end{proof}

\begin{Proposition}\label{p:exact-l} For $\chi_{\gamma^{-1}}$, there is an exact form $l_{\gamma^{-1}}$ such that $\chi_{\gamma^{-1}}={\rm d}l_{\gamma^{-1}}$, where
\begin{gather*}
l_{\gamma^{-1}}=\frac{1}{2}\log|\gamma'|^2\bigg(\dot{f}_z\circ\gamma +\dot{f}_z+2\frac{\dot{c}(\gamma)}{c(\gamma)}\bigg)
-\big(\log|c(\gamma)|^2+2+2\log 2\big)\big(\dot{f}_z\circ\gamma-\dot{f}_z\big).
\end{gather*}
\end{Proposition}

\begin{proof}
Recall
\begin{gather*}
{\chi}= \delta\xi + \delta(\phi_z \lambda \,{\rm d}z - \phi_{\bar{z}} \lambda\, {\rm d}\bar{z}) + L_{\mu} \check{\theta}.
\end{gather*}
Then, by Lemmas \ref{l:l3} and \ref{l:l4}, we have
\begin{align*}
\chi_{\gamma^{-1}}={}& -2\big(\dot{f}_{z\bar{z}}\circ\gamma\overline{\gamma'}-\dot{f}_{z\bar{z}}\big){\rm d}\bar{z}-\phi {\rm d}\big(\dot{f}_z\circ\gamma-\dot{f}_z\big)+\log|\gamma'|^2 {\rm d}\big(\dot{f}_z\circ\gamma\big)
\\
&+\bigg(\phi-\frac{1}{2}\log|\gamma'|^2-2\log2-\log|c(\gamma)|^2\bigg){\rm d}\big(\dot{f}_z\circ\gamma-\dot{f}_z\big)
\\
&-2\bigg(\dot{f}_{zz}\circ\gamma\gamma'-\dot{f}_{zz}-\lambda\frac{\gamma''}{\gamma'} \bigg){\rm d}z
+\bigg( \frac{1}{2}\big(\dot{f}_z\circ\gamma+\dot{f}_z\big)+\frac{\dot{c}(\gamma)}{c(\gamma)}-\lambda \bigg){\rm d}\log|\gamma'|^2
\\
&-\lambda \bigg(\frac{\gamma''}{\gamma'} {\rm d}z - \frac{\overline{\gamma''}}{\overline{\gamma'}}{\rm d}\bar{z}\bigg).
\end{align*}
On the right hand side of the above equality, the terms involving $\lambda$ cancel each other and the terms
involving $\phi$ also cancel each other. Now let us rewrite $\chi_{\gamma^{-1}}$ changing the order of terms as follows.
\begin{align*}
\chi_{\gamma^{-1}}={}&\frac12 {\rm d}\log|\gamma'|^2 \big(\dot{f}_z\circ \gamma +\dot{f}_z\big)
+\frac12 \log|\gamma'|^2 {\rm d}\big(\dot{f}_z\circ \gamma +\dot{f}_z\big)
\\
&-\big(\log|c(\gamma)|^2+2\log 2\big)\, {\rm d} \big(\dot{f}_z\circ \gamma - \dot{f}_z\big)
\\
&-2\big(\dot{f}_{z\bar{z}}\circ\gamma\overline{\gamma'}-\dot{f}_{z\bar{z}}\big){\rm d}\bar{z}
-2\big(\dot{f}_{zz}\circ\gamma\gamma'-\dot{f}_{zz} {\rm d}\big)z
+\frac{\dot{c}(\gamma)}{c(\gamma)} {\rm d}\log|\gamma'|^2
\\
={}&\frac12 {\rm d}\log|\gamma'|^2 \big(\dot{f}_z\circ \gamma +\dot{f}_z\big)
+\frac12 \log|\gamma'|^2 {\rm d}\big(\dot{f}_z\circ \gamma +\dot{f}_z\big)
\\
&-\big(\log|c(\gamma)|^2+2\log 2\big) \,{\rm d} \big(\dot{f}_z\circ \gamma - \dot{f}_z\big)
\\
&-2\dot{f}_{\bar{z}}\frac{\gamma''}{\gamma'} {\rm d}\bar{z}
-\big(\dot{f}_z\circ \gamma+\dot{f}_z\big)\frac{\gamma''}{\gamma'} {\rm d}z -\frac{\dot{c}(\gamma)}{c(\gamma)} \bigg(\frac{\gamma''}{\gamma'} {\rm d}z - \frac{\overline{\gamma''}}{\overline{\gamma'}} {\rm d}\bar{z}\bigg),
\end{align*}
where we used the equalities \eqref{e:dot-f-app} and \eqref{e:cc}.
Finally we can check that the exact form $l_{\gamma^{-1}}$ satisfying $\chi_{\gamma^{-1}}= {\rm d}l_{\gamma^{-1}}$
is given by
\begin{gather*}
l_{\gamma^{-1}}=\frac{1}{2}\log|\gamma'|^2\bigg(\dot{f}_z\circ\gamma +\dot{f}_z+2\frac{\dot{c}(\gamma)}{c(\gamma)}\bigg)
-\big(\log|c(\gamma)|^2+2+2\log 2\big)\big(\dot{f}_z\circ\gamma-\dot{f}_z\big).
\end{gather*}
For this, we use the following equality
\begin{gather*}
\dot{\gamma}' +\gamma'' \dot{f} = \dot{f}_z\circ \gamma \gamma' -\gamma' \dot{f}_z,
\end{gather*}
which follows from $\gamma^{\vep\mu}\circ f^{\vep\mu}= f^{\vep\mu} \circ \gamma$.
\end{proof}

By Proposition \ref{p:exact-l}, we have
\begin{gather}\label{e:del-xi}
\big\langle \delta\xi+\delta(\phi_z \lambda \,{\rm d}z - \phi_{\bar{z}} \lambda\, {\rm d}\bar{z}) +L_{\mu}\check{\theta}, L_1-L_2 \rangle =\langle {\rm d}l, L_1-L_2\big\rangle
=\langle l, \pa'L_1- \pa'L_2 \rangle.
\end{gather}
Since
\begin{gather*}
L_{\mu}\check{u}=L_{\mu}\delta \check{\theta}=\delta L_{\mu}\check{\theta}
=\delta \chi=\delta {\rm d}l={\rm d}\delta l,
\end{gather*}
we have
\begin{gather}\label{e:L-check-u}
\langle L_{\mu}\check{u}, W_1-W_2\rangle
=\langle\delta l, \pa'W_1-\pa'W_2\rangle
=\langle\delta l, V_1-V_2 \rangle
=\langle l, \pa''V_1-\pa''V_2 \rangle.
\end{gather}
From \eqref{e:der-var1}, \eqref{e:del-xi}, and \eqref{e:L-check-u}, it follows that
\begin{align*}
L_{\mu}\check{S}[\phi]
={}&\frac{\rm i}{2}\bigl(\big\langle \big(\big(2\phi_{zz}-\phi^2_z\big)\mu -2\phi_{z\bar{z}} \lambda\big)\, {\rm d}z\wedge {\rm d}\bar{z}, F_1-F_2 \big\rangle
\\
& -\langle \delta(\phi_z \lambda \,{\rm d}z - \phi_{\bar{z}} \lambda\, {\rm d}\bar{z} - \xi), L_1-L_2 \rangle -\langle L_{\mu}\check{\theta}, L_1-L_2 \rangle+\langle L_{\mu}\check{u}, W_1-W_2 \rangle\bigr)
\\
={}& \frac{\rm i}{2}\bigl(\big\langle \big(\big(2\phi_{zz}-\phi^2_z\big)\mu -2\phi_{z\bar{z}} \lambda \big)\, {\rm d}z\wedge {\rm d}\bar{z}, F_1-F_2 \big\rangle-\langle l, \pa'L_1-\pa'L_2-\pa''V_1+\pa''V_2 \rangle\bigr)
\\
={}&\frac{\rm i}{2}\big\langle \big(\big(2\phi_{zz}-\phi^2_z\big)\mu -2\phi_{z\bar{z}} \lambda \big)\, {\rm d}z\wedge {\rm d}\bar{z}, F_1-F_2 \big\rangle.
\end{align*}
This completes the proof of Theorem \ref{t:firstvariation}.

\begin{Theorem}
For a smooth family of conformal metrics ${\rm e}^{\phi^{\vep\mu}(z^\vep)}|{\rm d}z^\vep|^2$ on $X^\vep$,
\begin{gather*}
L_{\mu} {S}[\phi]=\int_{\Gamma\backslash \Omega} \big(\big(2\phi_{zz}-\phi^2_z\big)\mu +(1+K_\phi) {\rm e}^{\phi} \lambda \big)\, {\rm d}^2z,
\end{gather*}
where $\lambda=\dot{\phi}+\phi_z \dot{f} +\dot{f}_z$ with $\dot{\phi}=\frac{\rm d}{{\rm d}\vep}|_{\vep=0} \phi^{\vep\mu}$ and $K_\phi=-2\phi_{z\bar{z}} {\rm e}^{-\phi}$.
\end{Theorem}

\begin{proof}
We proved the formula for $L_{\mu} \check{S}$ in Theorem \ref{t:firstvariation}. For the remaining part,
it is easy to~see
\begin{gather*}
\frac{\partial}{\partial\vep}\bigg|_{\vep=0} \bigg( {\rm e}^{\phi^{\vep\mu}\circ f^{\vep\mu}} \frac{\rm i}{2} {\rm d}f^{\vep\mu}\wedge {\rm d}\bar{f}^{\vep\mu} \bigg) = {\rm e}^\phi \lambda\, {\rm d}^2z.
\end{gather*}
Hence,
\begin{align*}
L_\mu S[\phi] &=\int_{\Gamma\backslash \Omega} \big( \big(2\phi_{zz}-\phi^2_z\big)\mu -2\phi_{z\bar{z}} \lambda \big)\, {\rm d}^2z + \int_{\Gamma\backslash \Omega} {\rm e}^\phi \lambda\, {\rm d}^2z
\\
&= \int_{\Gamma\backslash \Omega} \big(\big(2\phi_{zz}-\phi^2_z\big)\mu +\big(1 -2\phi_{z\bar{z}} {\rm e}^{-\phi} \big) {\rm e}^{\phi}\lambda \big)\, {\rm d}^2z.
\end{align*}
This completes the proof.
\end{proof}

\subsection*{Acknowledgements}
This work was partially supported by Samsung Science and Technology Foundation under Project Number SSTF-BA1701-02.
The author thank referees for their helpful comments and suggestions which improve the exposition of the paper.

\pdfbookmark[1]{References}{ref}
\LastPageEnding


\begin{thebibliography}{99}
\footnotesize\itemsep=0pt

\bibitem{3}
Ahlfors L.V., Some remarks on {T}eichm\"uller's space of {R}iemann surfaces,
 \href{https://doi.org/10.2307/1970309}{\textit{Ann. of Math.}} \textbf{74} (1961), 171--191.

\bibitem{J}
Jost J., Compact {R}iemann surfaces. An introduction to contemporary
 mathematics, 3rd ed., \textit{Universitext}, \href{https://doi.org/10.1007/978-3-540-33067-7}{Springer-Verlag}, Berlin, 2006.

\bibitem{Kras00}
Krasnov K., Holography and {R}iemann surfaces, \href{https://doi.org/10.4310/ATMP.2000.v4.n4.a5}{\textit{Adv. Theor. Math. Phys.}}
 \textbf{4} (2000), 929--979, \href{https://arxiv.org/abs/hep-th/0005106}{arXiv:hep-th/0005106}.

\bibitem{Kras-Sch08}
Krasnov K., Schlenker J.-M., On the renormalized volume of hyperbolic
 3-manifolds, \href{https://doi.org/10.1007/s00220-008-0423-7}{\textit{Comm. Math. Phys.}} \textbf{279} (2008), 637--668,
 \href{https://arxiv.org/abs/math.DG/0607081}{arXiv:math.DG/0607081}.

\bibitem{M-P}
McIntyre A., Park J., Tau function and {C}hern--{S}imons invariant,
 \href{https://doi.org/10.1016/j.aim.2014.05.005}{\textit{Adv. Math.}} \textbf{262} (2014), 1--58, \href{https://arxiv.org/abs/1209.4158}{arXiv:1209.4158}.

\bibitem{1}
Park J., Takhtajan L.A., Teo L.-P., Potentials and {C}hern forms for
 {W}eil--{P}etersson and {T}akhtajan--{Z}ograf metrics on moduli spaces,
 \href{https://doi.org/10.1016/j.aim.2016.10.002}{\textit{Adv. Math.}} \textbf{305} (2017), 856--894, \href{https://arxiv.org/abs/1508.02102}{arXiv:1508.02102}.

\bibitem{PT17}
Park J., Teo L.-P., Liouville action and holography on quasi-{F}uchsian
 deformation spaces, \href{https://doi.org/10.1007/s00220-018-3164-2}{\textit{Comm. Math. Phys.}} \textbf{362} (2018), 717--758,
 \href{https://arxiv.org/abs/1709.08787}{arXiv:1709.08787}.

\bibitem{Sam}
Sampson J.H., Some properties and applications of harmonic mappings,
 \href{https://doi.org/10.24033/asens.1344}{\textit{Ann. Sci. \'Ecole Norm. Sup.~(4)}} \textbf{11} (1978), 211--228.

\bibitem{2}
Takhtajan L.A., Teo L.-P., Liouville action and {W}eil--{P}etersson metric on
 deformation spaces, global {K}leinian reciprocity and holography,
 \href{https://doi.org/10.1007/s00220-003-0878-5}{\textit{Comm. Math. Phys.}} \textbf{239} (2003), 183--240,
 \href{https://arxiv.org/abs/math.CV/0204318}{arXiv:math.CV/0204318}.

\bibitem{Tromba}
Tromba A.J., Teichm\"uller theory in {R}iemannian geometry, \textit{Lectures in
 Mathematics ETH Z\"urich}, \href{https://doi.org/10.1007/978-3-0348-8613-0}{Birkh\"auser Verlag}, Basel, 1992.

\bibitem{W89}
Wolf M., The {T}eichm\"uller theory of harmonic maps, \href{https://doi.org/10.4310/jdg/1214442885}{\textit{J.~Differential
 Geom.}} \textbf{29} (1989), 449--479.

\bibitem{TZ88-1}
Zograf P.G., Takhtadzhyan L.A., On {L}iouville's equation, accessory parameters, and the geometry of {T}eich\-m\"uller space for {R}iemann surfaces of genus~$0$, \href{https://doi.org/10.1070/SM1988v060n01ABEH003160}{\textit{Math. USSR-Sb.}} \textbf{60} (1988), 143--161.

\bibitem{TZ88-2}
Zograf P.G., Takhtadzhyan L.A., On uniformization of {R}iemann surfaces and the {W}eil--{P}etersson metric on {T}eichm\"uller and {S}chottky spaces, \href{https://doi.org/10.1070/SM1988v060n02ABEH003170}{\textit{Math. USSR-Sb.}} \textbf{60} (1988), 297--313.

\end{thebibliography}
\end{document}